\documentclass[10pt]{amsart}
\newcounter{pecounter}
\usepackage{pdfsync}\synctex=1

\usepackage{hyperref}
\usepackage{amsfonts,amssymb,stmaryrd,amsmath}
\usepackage{enumitem,moreenum}
\usepackage{multirow}

\usepackage{comment}
\usepackage{graphicx,tikz,tabularx}
\usetikzlibrary{cd,arrows}
\usetikzlibrary{matrix,arrows,decorations.markings}
\usepackage{amsthm,mathtools,frcursive,xcolor}

\usepackage[utf8]{inputenc}
\usepackage[T1]{fontenc}

\usepackage[bbgreekl]{mathbbol}
\DeclareSymbolFontAlphabet{\amsmathbb}{AMSb}

\ifnum \arabic{pecounter} = 0
\excludecomment{details}
\fi

\newcounter{numstep}
\newcommand{\step}[1]{\bigskip

\noindent {\bf Step \arabic{numstep}. #1}
\stepcounter{numstep}

\medskip
}

\newcommand\longto{{\longrightarrow}}

\newcommand{\fS}{\mathfrak{S}}

\newcommand{\clmn}{c_{\lambda,\mu}^\nu}
\newcommand{\clmnp}{c_{\lambda(p,1),\mu(p,1)}^{\nu(p,1)}}
\newcommand{\clmnpq}{c_{\lambda(p,q),\mu(p,q)}^{\nu(p,q)}}
\newcommand{\Clmn}{c_{\lambda,\mu,\nu}}

\newcommand{\Proj}{{\mathrm{Proj}}}

\newcommand\PP{\mathbb P}
\newcommand\CC{\mathbb{C}}\newcommand\G{\mathbb{G}}\newcommand\QQ{\mathbb{Q}}
\newcommand\NN{\mathbb{N}}\newcommand\ZZ{\mathbb{Z}}

\makeatletter
\def\revddots{\mathinner{\mkern1mu\raise\p@\vbox{\kern7\p@\hbox{.}}\mkern2mu\raise4\p@\hbox{.}\mkern2mu\raise7\p@\hbox{.}\mkern1mu}}
\makeatother

\newcommand\Li{{\mathcal {L}}}

\newcommand\cO{{\mathcal {O}}}

\newcommand\quot{/\hspace{-0.4ex}/}
\newcommand{\oalpha}{{\overline{\alpha}}}
\newcommand{\oI}{{\overline{I}}}
\newcommand{\oJ}{{\overline{J}}}
\newcommand{\oK}{{\overline{K}}}

\newcommand\lgl{{\mathfrak{gl}}}

\newcommand{\lstab}{{\mathfrak{stab}}}
\newcommand{\ltrans}{{\mathfrak{trans}}}
\newcommand{\sgen}{{s_{\mathrm{gen}}}}
\newcommand{\tgen}{{t_{\mathrm{gen}}}}

\newcommand\C{{\mathbb C}}
\newcommand\Z{{\mathbb Z}}

\newcommand\cL{{\mathcal{L}}}\newcommand\cF{{\mathcal{F}}}

\newcommand{\rss}{{\rm ss}}
\newcommand{\ac}{{\mathcal{AC}}}
\newcommand{\acgx}{\ac^G(X)}
\newcommand{\modx}{\Eqd(X,G)}
\newcommand{\modxq}{\Eqd(X(q),G(q))}

\newcommand\Ho{{\operatorname H}^0}
\newcommand\Hom{{\operatorname{Hom}}}
\newcommand{\rk}{{\operatorname{rk}}}
\newcommand{\im}{{\operatorname{Im}}}
\newcommand\GL{{\operatorname{GL}}}

\newcommand\Fl{{\operatorname{Fl}}}
\newcommand\Eqd{{\operatorname{Eqd}}}

\newcommand\Pic{{\operatorname{Pic}}}

\newcommand\Gr{{\operatorname{Gr}}}

\newtheorem{notation}{Notation}
\newtheorem{exple}{Example}

\newtheorem{prop}{Proposition}

\newtheorem{theo}[prop]{Theorem}
\newtheorem{lemma}[prop]{Lemma}

\newtheorem{coro}[prop]{Corollary}

\newcounter{details}
\ifnum \arabic{details} = 0
\excludecomment{pournous}
\fi
\ifnum \arabic{details} > 0

\fi

\ifnum \arabic{details} = 2
\usepackage{showkeys}
\fi

\ifnum \arabic{details} > 0

\fi

\ifnum \arabic{details} = 0 
\excludecomment{pe}
\excludecomment{thomas} 
\excludecomment{lucas}
\fi

\newcommand\inv{{^{-1}}}

\newcommand\DynkinNodeSize{2mm}
\newcommand\DynkinArrowLength{3mm}
\tikzset{
  dnode/.style={
    circle,
    inner sep=0pt,
    minimum size=\DynkinNodeSize,
    fill=white,
    draw},
  middlearrow/.style={
    decoration={markings,
      mark=at position 0.6 with
      {\draw (0:0mm) -- +(+135:\DynkinArrowLength); \draw (0:0mm) -- +(-135:\DynkinArrowLength);},
    },
    postaction={decorate}
  },
  leftrightarrow/.style={
    decoration={markings,
      mark=at position 0.999 with
      {
      \draw (0:0mm) -- +(+135:\DynkinArrowLength); \draw (0:0mm) -- +(-135:\DynkinArrowLength);
      },
      mark=at position 0.001 with
      {
      \draw (0:0mm) -- +(+45:\DynkinArrowLength); \draw (0:0mm) -- +(-45:\DynkinArrowLength);
      },
    },
    postaction={decorate}
  },
  sedge/.style={
  },
  dedge/.style={
    middlearrow,
    double distance=0.5mm,
  },
  tedge/.style={
    middlearrow,
    double distance=1.0mm+\pgflinewidth,
    postaction={draw}, 
  },
  infedge/.style={
    leftrightarrow,
    double distance=0.5mm,
  }
}
\begin{document}
\title{Bidilation of small Littlewood-Richardson coefficients}

\author[P.-E. Chaput]{Pierre-Emmanuel Chaput}
\address{Universit\'e de Lorraine, CNRS, Institut \'Elie Cartan de Lorraine, UMR 7502, Vandoeu\-vre-l\`es-Nancy, F-54506, France}
\email{pierre-emmanuel.chaput@univ-lorraine.fr}

\author[N. Ressayre]{Nicolas Ressayre}
\address{Université Claude Bernard Lyon I, Institut Camille Jordan (ICJ), UMR CNRS 5208, 43 boulevard du 11 novembre 1918,
69622 Villeurbanne CEDEX}
\email{ressayre@math.univ-lyon1.fr}

\thanks{This work was supported by the ANR GeoLie project, of the French Agence Nationale de la Recherche.}

\begin{abstract}
The Littlewood-Richardson coefficients $c_{\lambda,\mu}^\nu$ are the multiplicities in the
tensor product decomposition of two irreducible representations of the
general linear group $\GL(n,\CC)$. 
They are parametrized by the triples of partitions $(\lambda,\mu,\nu)$
of length at most $n$.
By the so-called Fulton conjecture, if $c_{\lambda,\mu}^\nu=1$ then
$c_{k\lambda,k\mu}^{k\nu}=1$, for any $k\geq 0$.
Similarly, as proved by Ikenmeyer or Sherman,  if $c_{\lambda,\mu}^\nu=2$ then
$c_{k\lambda,k\mu}^{k\nu}=k+1$, for any $k\geq 0$.

Here, given a partition $\lambda$, we set 
$$
\lambda(p,q)=p(q\lambda')',
$$
where prime denotes the conjugate partition.
We observe that Fulton's conjecture implies that   if $c_{\lambda,\mu}^\nu=1$ then
$c_{\lambda(p,q),\mu(p,q)}^{\nu(p,q)}=1$, for any $p,q\geq 0$.
Our main result is that if $c_{\lambda,\mu}^\nu=2$ then 
$c_{\lambda(p,q),\mu(p,q)}^{\nu(p,q)}$ is the binomial 
$\binom{p+q}{q}$, for any $p,q\geq 0$.
\end{abstract}

\maketitle

\section{Introduction}

Fix an $n$-dimensional vector space $V$.
Given a partition  $\lambda=(\lambda_1\geq\cdots\geq\lambda_n\geq 0)$ with
$\lambda_i\in\NN$, let $S^\lambda V$ be the corresponding Schur module,
that is the irreducible $\GL(V)$-module of highest weight $\sum \lambda_i \epsilon_i$ (notation as in \cite{bourb}). 
This paper is concerned by the Littlewood-Richardson coefficients $c_{\lambda,\mu}^\nu$ 
defined by
\begin{equation}
  \label{eq:17}
  S^\lambda V\otimes S^\mu V \simeq \bigoplus_{\nu}
  \C^{c_{\lambda,\mu}^\nu} \otimes S^\nu V,
\end{equation}
where $\C^{c_{\lambda,\mu}^\nu}$ is a multiplicity space.
Given a partition $\lambda$ as above, we set 
$$
\lambda(p,q)=(p\lambda_1,\ldots,p \lambda_1,p \lambda_2,\ldots,p \lambda_2,\ldots,p \lambda_n,\ldots, p \lambda_n)
$$
where each part is repeated $q$ times.
Fulton's conjecture (see \cite{KTW,belkale:fulton,ressayre:fulton} for various proofs) can be restated as:

\begin{theo}
\label{th:Fulton}
  If $c_{\lambda,\mu}^\nu=1$ then, for any positive $p$ and $q$, we have
$$
c_{\lambda(p,q),\mu(p,q)}^{\nu(p,q)}=1.
$$
\end{theo}

The ordinary formulation of Fulton's conjecture corresponds to the case
$q=1$. The general case follows from the equality
$c_{\lambda',\mu'}^{\nu'} = \clmn$, where $\lambda'$ denotes the conjugated partition of $\lambda$.
This ordinary version has an extension to the case $c_{\lambda,\mu}^\nu=2$,
see \cite{ikenmeyer}
and \cite[Theorem 1.1 and Corollary 9.4]{sherman} for a generalization in the context of quivers:

\begin{theo}
\label{th:Iken}
  If $\clmn=2$ then, for any positive integers $p,q$, we
  have
$$ 
c_{\lambda(p,1),\mu(p,1)}^{\nu(p,1)}=p+1
\mbox{\ \ and\ \ }
c_{\lambda(1,q),\mu(1,q)}^{\nu(1,q)}=q+1.
$$
\end{theo}

Our main result is an extension of Theorem~\ref{th:Iken} in the spirit
of Theorem~\ref{th:Fulton}:

\begin{theo}
 \label{main-theo}
 If $\clmn=2$ then, for any positive integers $p,q$, we have
 $$
 c_{\lambda(p,q),\mu(p,q)}^{\nu(p,q)}=\binom{p+q}{q}.
 $$
\end{theo}

Here, $\binom{p+q}{q}$ stands for the binomial.
Ikenmeyer proved Theorem~\ref{th:Iken} using convex geometry and integral points counting, whereas we use Geometric Invariant Theory.
An example of a triple of partitions $(\lambda,\mu,\nu)$ such that
$c_{\lambda(p,q),\mu(p,q)}^{\nu(p,q)}=\begin{pmatrix} p+q\\q \end{pmatrix}$ is given in
\cite[Example 6.2]{KTW}.

\bigskip

The main idea for the value $\clmnpq=\binom{p+q}{p}$ is the following (although the proof of the following claims
is less direct than what is presented in this introduction). First, letting $G=\GL_n$,
we interpret the coefficient $\clmn$ as the dimension of a space
of $G$-invariant sections of a line bundle $\cL$ on the product $X$ of three flag varieties under the group $G$.
The coefficient $\clmnp$ is then simply the dimension of $\Ho(X,\cL^{\otimes p})^G$. The coefficient $\clmnpq$
in turn has a geometrical definition dilating the flag variety $X$. More precisely, we replace $X$ by $X(q)$ which is
a product of partial flag varieties for $G(q):=\GL_{nq}$, and we replace $\cL$ by some line bundle $\cL(q)$. We get:
$$
\clmnpq = \dim \Ho(X(q),\cL(q)^{\otimes p})^{G(q)}.
$$

Using properties of the Horn cone proved in \cite{dw,ressayre-birational},
we observe that if $(\lambda,\mu,\nu)$ is not general, then
$\clmn$ is in fact the product of two Littlewood-Richardson coefficients for smaller linear groups, and we conclude
by induction.

By results of Ikenmeyer and Sherman \cite{ikenmeyer,sherman2015geometric}, the polarized GIT-quotient
$X^{\rm ss}(\Li) \quot G$ is isomorphic to $(\PP^1,\cO_{\PP^1}(1))$. The equality $\clmnp=p+1$ is explained by
the equality $\dim \Ho(\PP^1,\cO_{\PP^1}(p))=p+1$.

We produce in \eqref{eq:defiota} an inclusion of $X^q$ in $X(q)$.
If $(\lambda,\mu,\nu)$ is general, then the codimension of a general $G$-orbit in $X$ has codimension $1$, and
we show that the codimension of a general $G(q)$-orbit in $X(q)$ will have codimension $q$, from which we deduce
that the restriction induces an isomorphism
$$
\Ho(X(q),\Li(q)^{\otimes p})^{G(q)} \longto
\Ho(X^q,\boxtimes^q \Li^{\otimes p})^{N_{G(q)}(X^q)},
$$
where $N_{G(q)}(X^q)$ denotes the stabilizer of $X^q$ in $G(q)$.
Therefore, understanding the GIT-quotient $X(q)^\rss(\Li(q)) \quot G(q)$ comes down to understanding the GIT-quotient
$(X^q)^\rss(\Li(q))  \quot N_{G(q)}(X^q)$.
The action of $N_{G(q)}(X^q)$ on $X^q$ is given by the action of $G^q$ on $X^q$
and the permutation of the $q$ factors, from which it follows that
the quotient $X(q)^\rss(\Li(q))  \quot G(q)$ is isomorphic to the quotient of $\left ( X ^\rss(\Li)  \quot G \right )^q$ by the symmetric
group $\fS_q$, which is ${(\PP^1)}^q \quot \fS_q$, namely $\PP^q$.

It follows that the polarized GIT-quotient $X(q)^\rss(\Li(q))  \quot G(q)$ is $(\PP^q,\cO_{\PP^q}(1))$
(see Corollary \ref{coro:Pq}), and taking the $p$-th
power of the polarization, we obtain our binomial coefficient as the number
$\dim \Ho(\PP^q,\cO_{\PP^q}(p))$.

\setcounter{tocdepth}{1}
\tableofcontents

\section{$G$-ample cone of flag varieties}

\subsection{GIT-quotient}

Let $G$ be a complex connected reductive group acting on an irreducible projective variety $X$.  
Let $\Pic^G(X)$ denote the group of $G$-linearized line bundles on $X$.
For $\Li\in\Pic^G(X)$, $\Ho(X,\Li)$ denotes the $G$-module of regular sections of $\Li$ and
 $\Ho(X,\Li)^G$ denotes  the subspace of $G$-equivariant sections.
For any $\Li\in\Pic^G(X)$, 
we set 
$$
X^{\rm ss}(\Li,G)= X^{\rm ss}(\Li)=\{x\in X \,:\,\exists n>0{\rm\ and\ }\sigma\in\Ho(X,\Li^{\otimes n})^G 
\  {\rm s.\,t.}\ \sigma(x)\neq 0\}.
$$
Note that this definition of $X^{\rm ss}(\Li)$ coincides with that of \cite[Definition 1.7]{GIT} if $\Li$ is ample
but not in general.


Assuming that $X^{\rm ss}(\Li)$ is not empty, 
consider the following projective variety
\begin{equation}
 \label{equa:git-quotient}
 X^{\rm ss}(\Li)\quot G:={\rm Proj}\left (\bigoplus_{n\geq
     0}\Ho(X,\Li^{\otimes n})^G\right ),
 \vspace{-3mm}
\end{equation}
and the natural $G$-invariant morphism
$$
\pi\,:\,X^{\rm ss}(\Li)\longto X^{\rm ss}(\Li)\quot G.
$$
If $\Li$ is ample then $\pi$ is a good quotient and, in particular,
the points in $X^{\rm ss}(\Li)\quot G$
correspond to the closed $G$-orbits in $X^{\rm ss}(\Li)$.

\subsection{The $G$-ample cone}
\label{sec:ample-GIT-cone}

We assume here that $\Pic^G(X)$ has finite rank and we consider the rational vector
space $\Pic^G(X)_\QQ:=\Pic^G(X)\otimes_\ZZ\QQ$.
Since $X^{\rm ss}(\Li)=X^{\rm ss}(\Li^{\otimes n})$ for any positive integer $n$, 
 $X^{\rm ss}(\Li)$ can be defined for any element $\Li$ in  $\Pic^G(X)_\QQ$.
The set of ample line bundles in $\Pic^G(X)$ generates an open convex cone 
$\Pic^G(X)_\QQ^+$ in  $\Pic^G(X)_\QQ$.  The following cone was defined in 
\cite{DH} and is called the \emph{$G$-ample cone}:
\begin{equation}
 \label{equa:def-ac} 
 \ac^G(X):=\{\Li\in\Pic^G(X)_\QQ^+\,:\,X^{\rm ss}(\Li)\neq\emptyset\}.
\end{equation}
Accordingly,
a line bundle $\Li\in \Pic^G(X)$ is said to be {\it $G$-ample} if
$X^{\rm ss}(\Li)$ is not empty.
Since the product of two nonzero $G$-equivariant sections of two line 
bundles is a nonzero $G$-equivariant section of the tensor product of the two line bundles, 
$\ac^G(X)$ is convex: see \cite[Propositions 3.1.2, 3.1.3 and Definition 3.2.1]{DH}.

\medskip

Let $\modx$ denote the minimal codimension of 
$G$-orbits in $X$.  
By \cite[Proposition 4.1]{algo}, the expected quotient dimension is the maximal dimension of the quotients:

\begin{prop}
\label{prop:modmax}
Assume that $X$ is smooth. The maximal dimension of varieties
$X^{\rm ss}(\Li)\quot G$ for $\Li\in\Pic^G(X)$ is equal to
$\modx$. Moreover, for any 
$\Li$ in the relative interior of $\acgx$, $\dim(X^{\rm ss}(\Li)\quot
G)=\modx$.
\end{prop}

\subsection{Restriction to the $\tau$-fixed locus}

Let $\tau:\C^* \to G$ be a one parameter
subgroup. Let $C$ be an irreducible component of the $\tau$-fixed
point set $X^\tau$.
Let $\Li$ be an ample $G$-linearized line bundle.  

Since the centralizer $G^\tau$ of $\tau$ is connected, it acts on $C$. Moreover, 
by Luna (see e.g. \cite[Proposition~8]{GITEigen}), we have
\begin{equation}
C^\rss(\Li_{|C},G^\tau) = X^\rss(\Li,G) \cap C.\label{eq:1}
\end{equation}
Here $\Li_{|C}$ stands for the restriction of $\Li$ to $C$.
Thus, the following defines a morphism
\begin{equation}
\label{def-pi}
p\,:\,C^\rss(\Li_{|C},G^\tau)\quot G^\tau\longto X^\rss(\Li,G)\quot G\,.
\end{equation}

\begin{lemma}
\label{lem:pifinite} The morphism
$p$ is finite on its image.
\end{lemma}

\begin{proof}
The quotient map $\pi\,:\, X^{\rm ss}(\Li)\longto X^\rss(\Li,G)\quot
G$ being affine, 
this follows from \cite[Theorem 2]{luna:invariants}.
%
\end{proof}

\subsection{Notations about flag varieties and line bundles on them}
\label{sec:flag}

Let $V$ be a  finite dimensional complex vector space.
Fix a basis $(e_1,\dots,e_n)$ of $V$ and identify the linear group $\GL(V)$ with
$\GL_n(\CC)$. 
Let $T \subset \GL(V)$ (resp. $B \subset \GL(V)$) be the maximal torus
(resp. Borel subgroup) containing 
all diagonal (resp. upper triangular) matrices. 
Let $\epsilon_i\,:\,T\longto\CC^*$ denote the character mapping $t\in T$ on
its $i$-th diagonal entry.
Note that $(\epsilon_i)_{1\leq i\leq n}$ forms a $\ZZ$-basis of the
character group $X(T)$ of $T$.
Moreover, the set $X(T)^+$ of dominant characters identifies with 
$$
\Lambda_n^+=\{(\lambda_1,\ldots,\lambda_n) \in\ZZ^n {\rm\ with\ }
\lambda_1 \geq \lambda_2 \geq \cdots \geq \lambda_n\},
$$
by mapping $(\lambda_1,\ldots,\lambda_n)$ to $\sum_i\lambda_i\epsilon_i$.
Let $\varpi_i=\epsilon_1+\cdots+\epsilon_i$ be the $i$-th fundamental
weight. 

Given integers $0<a_1<\cdots<a_r<n$, we let $\Fl(a_1,\ldots,a_r;V)$
denote the corresponding partial flag variety:
$$
\Fl(a_1,\ldots,a_r;V)=\{V_1\subset\cdots\subset V_r\subset V\,:\,\dim(V_i)=a_i\}.
$$
This will also be denoted $\Fl(A;V)$ where $A=\{a_1,\ldots,a_r\}$.
The standard base point $\xi_0$ in $\Fl(a_1,\ldots,a_r;V)$ is defined
by letting $V_i$ be the span of $e_j$'s for $j\leq a_i$.
The stabilizer of $\xi_0$ in $\GL(V)$ is denoted by $P$. 
Moreover $\varpi_i$ extends to a character of $P$ if and only if $i
\in A$. 
In this case, this defines a $G$-linearized line bundle $\GL(V) \times^P \C_{-\varpi_i}$
on $\Fl(A;V)$, and its space of sections is $\wedge^i V^*$, as a $G$-representation.

More generally, given $(\lambda_1,\dots,\lambda_n)\in\Lambda_n^+$,
we define the line
bundle
\begin{equation}
 \label{equa:def-L-flag}
 \Li_\lambda=\GL(V) \times^P \C_{-\lambda} \mbox{ where }
 \lambda = \sum_{i=1}^{n} \lambda_i \epsilon_i = \sum_{i=1}^{n} (\lambda_i-\lambda_{i+1}) \varpi_i\,,
\end{equation}
with the convention $\lambda_{n+1}=0$.
It is well-defined if and only if $\lambda$ is a weight of $P$,
which means that
\begin{equation}
 \label{equa:line-defined}
 0<i<n \mbox{ and } \lambda_i>\lambda_{i+1} \Longrightarrow i \in A.
\end{equation}
Moreover, $\Li_\lambda$ is ample on $\Fl(A;n)$ if and only if this is an equivalence:
\begin{equation}
 \label{equa:line-ample}
 0<i<n \mbox{ and } \lambda_i>\lambda_{i+1} \Longleftrightarrow i \in A.
\end{equation}

Borel-Weil theorem says that
\begin{equation}
 \Ho(\Fl(A;V),\Li_\lambda) = S^\lambda V^*\,,
\end{equation}
where $S^\lambda$ is the Schur functor associated to $\lambda$.

Finaly, if $X=\Fl(A^1;V) \times \cdots \times \Fl(A^k;V)$ is a product of $k$ flag varieties, and
$\lambda^1,\dots,\lambda^k$ are in $\Lambda_n^+$  such that each pair $(\lambda^j,A^j)$ satisfies
\eqref{equa:line-defined}, we define the following line bundle on $X$:
\begin{equation}
 \label{equa:def-L-X}
 \Li_{(\lambda^1,\dots,\lambda^k)} = \Li_{\lambda^1} \boxtimes \cdots \boxtimes \Li_{\lambda^k}.
\end{equation}
Thus, $\Ho(X,\Li_{(\lambda^1,\dots,\lambda^k)})=S^{\lambda^1}V^*\otimes\cdots\otimes S^{\lambda^k}V^*$.

\subsection{The Horn cone of $\GL_n$}
\label{sec:horn}

Let $k$ be an integer. The cone inside ${(\QQ^n)}^k$ generated by the $k$-uples 
$(\lambda^j)$ in $\Lambda_n^+$ such that 
$(S^{\lambda^1} V^* \otimes \cdots \otimes S^{\lambda^k} V^*)^{\GL(V)} \neq \{0\}$
is called the \emph{Horn cone} and has a description that we now recall.

Let $I \subset \{1,\ldots,n\}$ be a subset with $r$ elements. The linear subspace $V_I \subset V$ generated
by the base vectors $e_i$ for $i \in I$ defines a $T$-fixed point in the Grassmannian $\G(r;V)$.
The cohomology class of the closure of the $B^-$-orbit through this point will be denoted by $\sigma_I$.
Here $B^-$ denotes the Borel subgroup of $\GL(V)$ consisting in lower
triangular matrices.

For $\lambda=(\lambda_1 \geq \cdots \geq \lambda_n)$ let $|\lambda| = \sum_i \lambda_i$.
By \cite{klyachko,belkale:horn}, the $k$-uple $(\lambda^j)$ belongs to the Horn
cone if and only if 
$$\sum_{j=1}^k|\lambda^j|=0,
$$
and the following holds for all integers
$r \in \{1,\ldots,n-1\}$ and all $k$-uples $(I^j)_{1 \leq j \leq k}$
of subsets of $\{1,\ldots,n\}$ with $r$ elements:
\begin{equation}
 \label{equa:horn}
 \sigma_{I^1} \cup \cdots \cup \sigma_{I^k} = [pt] \in {\operatorname H}^*(\G(r;V),\Z)
 \Longrightarrow
 \frac 1r \sum_{j=1}^k |\lambda^j_{I^j}| \leq \frac 1n \sum_{j=1}^k |\lambda^j|\,.
\end{equation}
Here,
$\lambda_I$ denotes the partition obtained by taking the parts $\lambda_i$ for $i \in I$. Moreover, by
\cite{KTW} (see also \cite{GITEigen}), for such $I^1,\ldots,I^k$,
each equation $\frac 1r \sum_{j=1}^k |\lambda^j_{I^j}| = \frac 1n \sum_{j=1}^k |\lambda^j|$
defines  a face of codimension $1$ in the Horn cone.

\subsection{The $G$-ample cone of products of flag varieties}
\label{sec:ample-flag}

We now assume that
\begin{equation}
X = \Fl(A^1;V) \times \cdots \times \Fl(A^k;V)\label{eq:defXk}
\end{equation}
is a product of flag varieties homogeneous under the group $G=\GL(V)$.

\begin{prop}
\label{prop:ample-cone}
A $G$-equivariant line bundle $\Li_{(\lambda^j)}$ on $X$ given by a
$k$-uple $(\lambda^j)_{1 \leq j \leq k}$ in $(\Lambda_n^+)^k$
is $G$-ample if and only if $\sum_{j=1}^k |\lambda^j|=0$, and 
\begin{equation}
 \label{equa:cone}
 \left \{
 \begin{array}{l}
  \lambda^j_i > \lambda^j_{i+1} \Longleftrightarrow i \in A^j \\
  \sigma_{I^1} \cup \cdots \cup \sigma_{I^k} = [pt] \Longrightarrow
  \frac 1r \sum_{j=1}^k |\lambda^j_{I^j}| \leq \frac 1n \sum_{j=1}^k |\lambda^j|
 \end{array}
 \right .
\end{equation}
\end{prop}
\begin{proof}
The first condition is equivalent to $\Li_{(\lambda^j)}$ being ample on $X$, and the second condition
is equivalent to
$\Ho(X,\Li)^G$ being non trivial: recall respectively \eqref{equa:line-ample} and \eqref{equa:horn}.
\end{proof}

\section{Geometric formulation of the main theorem}
\label{sec:setting}

For $\nu=(\nu_1 \geq \cdots \geq \nu_n)$ in $\Lambda_n^+$, set $\nu^\vee=(-\nu_n \geq \cdots \geq -\nu_1)$
such that $S^{\nu^\vee} V^*$ is the $\GL(V)$-representation
dual to $S^\nu V^*$. Moreover, let $\Clmn=c_{\lambda,\mu}^{\nu^\vee}$.
Since $\nu(p,q)^\vee=\nu^\vee(q,p)$,
our main Theorem \ref{main-theo} is equivalent to the implication
\begin{equation}
 \label{equa:main}
 \Clmn = 2 \Longrightarrow c_{\lambda(p,q),\mu(p,q),\nu(p,q)}=\binom{p+q}{q}.
\end{equation}

Thus, let $\lambda,\mu,\nu$ be such that $\Clmn=2$.
For $\eta\in\Lambda_n^+$, let $A(\eta)$ be the set $j \in
\{1,\ldots,n-1\}$ such that $\eta_j>\eta_{j+1}$. 
We fix the product
\begin{equation}
X=\Fl(A(\lambda);V) \times \Fl(A(\mu);V) \times
\Fl(A(\nu);V)\label{eq:defXlambda}
\end{equation}
of three partial flag varieties
and the ample line bundle $\Li:=\Li_{(\lambda,\mu,\nu)}$ on $X$, such that
$\Ho(X,\Li)=S^\lambda V^* \otimes S^\mu V^* \otimes S^\nu V^*$ (see Section \ref{sec:flag}).

Fix a $q$-dimensional vector space $E$. If $\cF=\Fl(a_1,\dots,a_s;V)$,
set $\cF(q)=\Fl(qa_1,\dots,qa_s;V\otimes E)$.
For $\eta\in\Lambda_n^+$, let $\eta(1,q)$ denote the partition with each part $\lambda_i$ repeated $q$ times.
Observe that if $\Li_\eta$ is a line
bundle (resp. ample line bundle) on $\cF=\Fl(A;V)$, then $\Li_{\eta(1,q)}$ is a line
bundle (resp. ample line bundle) on $\cF(q)$, by \eqref{equa:horn}.
Now, set $X(q) =\Fl(A(\lambda);V)(q) \times \Fl(A(\mu);V)(q) \times
\Fl(A(\nu);V)(q)$ and let $\Li(q)$ be the line bundle $\Li_{(\lambda(1,q),\mu(1,q),\nu(1,q))}$ on $X(q)$.
Then $c_{\lambda,\mu,\nu}=\dim\left (\Ho(X,\Li_{(\lambda,\mu,\nu)})^G\right)$
and
\begin{equation}
\label{equa:cxq}
c_{\lambda(p,q),\mu(p,q)\nu(p,q)}=\dim\left (\Ho(X(q),\Li(q)^{\otimes p})^{G(q)}\right ),
\end{equation}
where $G=\GL(V)$ and $G(q)=\GL(V\otimes E)$. 
Hence, our main theorem can be rephrased as the implication
\begin{equation}
 \label{equa:aim}
 \dim \left (\Ho(X,\Li)^G \right )=2 \ \Longrightarrow\  \dim \left (\Ho(X(q),\Li(q)^{\otimes p})^{G(q)}\right )=\binom{p+q}{p}.
\end{equation}

\section{Preparation of the proof of the main theorem}
\label{sec:preparation}

In this section, we fix $V$, $E$, $G$ and
$G(q)$ as in Section~\ref{sec:setting}. 
Fix also $k\geq 3$ and $A^1,\dots,A^k$ subsets of $\{1,\dots,n\}$.
Consider the varieties
$$
X=\Fl(A^1;V)\times\cdots\times \Fl(A^k;V),
$$
and
$$
X(q)=\Fl(qA^1;V\otimes E)\times\cdots\times \Fl(qA^k;V\otimes E).
$$

\subsection{A key construction}
\label{sec:defC}

A key observation is that $X^q$ embeds in
$X(q)$.
To make this embedding explicit, fix a basis $(f_1,\dots,f_q)$ of $E$.
Let $\tau$ be a regular diagonal (with respect to the fixed basis)
one-parameter subgroup of $\GL(E)$.

The group $\GL(E)$, and hence $\tau$, act on $V\otimes E$. 
A linear subspace $F\subset  V\otimes E$ is $\tau$-stable if and only
if there exist subspaces $(F_i)_{1\leq i\leq q}$ of $V$ such that 
$$
F=F_1\otimes\CC f_1\oplus\cdots\oplus F_q\otimes\CC f_q.
$$
Futhermore, the map 
$$
\begin{array}{ccl}
  \Gr(a;V)^q&\longto&\Gr(aq;V)\\
(F_i)_{1\leq i\leq q}&\longmapsto&F_1\otimes\CC f_1\oplus\cdots\oplus F_q\otimes\CC f_q,
\end{array}
$$
is an isomorphism onto an irreducible component of the $\tau$-fixed
point set.
Similarly, $\Fl(A;V)^q$ (resp. $X$) embeds in $\Fl(qA;V\otimes E)^q$
(resp. $X(q)$) as an
irreductible component of $\tau$-fixed points. Denote by
\begin{equation}
  \label{eq:defiota}
  \iota_q\,:\,X^q\longto C \subset X(q),
\end{equation}
the corresponding embedding and by $C$ its image.
It is equivariant for the action of $G(q)^\tau$, that is isomorphic to $G^q$.

\subsection{Expected quotient dimension of $X(q)$.}
\label{sec:modality}

For $X=\Fl(A^1) \times \cdots \times \Fl(A^k)$ as above, we introduce some more notation:

\begin{notation}
 \label{nota:trans}
 Given $x,y$ in $X$, write these elements as $x=(l^1,\ldots,l^k)$ and $y=(m^1,\ldots,m^k)$ with
 $l^j,m^j$ in $\Fl(A^j;V)$.
 \begin{itemize}
  \item Let $\ltrans(x,y)$ denote the subspace in $\lgl(V)$ of the endomorphisms such that for any $j \in \{1,\ldots,k\}$
  and any $i \in A^j$, $(l^j)_i$ is sent into $(m^j)_i$.
  \item Let $\lstab(x):=\ltrans(x,x)$.
  \item Let $\sgen$ be the dimension of the vector space
  $\lstab(x)$ for general $x$ in $X$.
  \item Let $\tgen$ be the dimension of the vector space
  $\ltrans(x,y)$ for general $(x,y)$ in $X^2$.
 \end{itemize}
\end{notation}

\begin{lemma}
 \label{lemm:stab-trans}
 With the above notation:
 \begin{enumerate}
  \item We always have $\tgen \leq \sgen$;
  \item If $\modx > 0$, then $\tgen \leq \sgen-1$.
 \end{enumerate}
\end{lemma}
\begin{proof}
The function $(x,y) \mapsto \dim \ltrans(x,y)$ is upper semi-continuous on $x$ and $y$, hence the first point.
For the second point, we assume $\tgen=\sgen$ and we prove that
$\modx=0$.
Let $U$ be the set of $(x,y)\in X^2$ such that
$\dim \ltrans(x,y)=\tgen$.
The theory of linear systems implies that 
\begin{equation}
{\mathcal E}:=\{(x,y,\xi)\in U\times
\lgl(V)\;:\; \xi\in \ltrans(x,y)\} \label{eq:defE}
\end{equation}
is a vector bundle on $U$.

Therefore, the set $\Sigma$ of pairs $(x,y) \in U$ such that
$\ltrans(x,y) \subset \{\det=0\} \subset \lgl(V)$ is closed in $U$.
Since $\lstab(x)=\ltrans(x,x)$ contains the identity map of $V$ for any $x\in
X$, $\Sigma$ does not intersect the diagonal $\Delta=\{(x,x)\;:\;x\in X\}$.

But, the assumption $\tgen=\sgen$ implies that $U$ intersects $\Delta$. Hence $\Sigma$ is a proper closed subset
of $U$. For any $(x,y)\in U \setminus \Sigma$, $\ltrans(x,y)$ intersects $\GL(V)$, so $x$ and $y$ belong to the
same $\GL(V)$-orbit. Let $p_1:X \times X \to X$ be the first projection. For $x$ in $p_1(U)$ and
$y$ in the open subset $p_1^{-1}(x) \cap U$ of $p_1^{-1}(x) \simeq X$, it follows that $x$ and $y$ are in the
same $G$-orbit. Thus the $G$-orbit through $x$ is dense in X, and $\modx=0$.
\end{proof}

\begin{prop}
\label{prop:modality}
With the above notation:
\begin{enumerate}
\item If $\modx=0$ then  $\modxq=0$;
\item If $\modx>0$ then  $\modxq\leq q^2(\modx-1)+q$. 
\end{enumerate}
\end{prop}
\begin{proof}
Let $(x_1,\ldots,x_q)\in X^q$ be general in $X^q$, and set
$y=\iota_q((x_1,\ldots,x_q))$. We are interested in $\lstab(y)$. The Lie algebra
$\lgl(V\otimes E)$ identifies with the set of $(q \times q)$-matrices
with entries in $\lgl(V)$. 
Accordingly, $\lstab(y)$ decomposes as
\begin{equation*}
 \label{equa:stab-y}
 \lstab(y) = \bigoplus_{1 \leq i,j \leq q} \ltrans(x_i,x_j) \otimes \Hom(\CC f_i,\CC f_j)
\end{equation*}
which implies
\begin{equation*}
 \label{equa:dim-stab}
 \dim \lstab(y) = \sum_{1 \leq i,j \leq q} \dim \ltrans(x_i,x_j) = q \sgen + (q^2-q) \tgen.
\end{equation*}

Assuming $\modx=0$, we deduce from Lemma \ref{lemm:stab-trans}(1) that
$\dim \lstab(y) \leq q^2 \sgen=q^2(\dim G - \dim X)$. It follows that the orbit $G(q) \cdot y$ has dimension at least
$q^2 \dim G - q^2 (\dim G-\dim X) = q^2 \dim X = \dim X(q)$, so that
$X(q)$ has expected quotient dimension $0$.

Assuming $\modx>0$, set $m=\modx$. We deduce from Lemma \ref{lemm:stab-trans}(2) that
$$
\begin{array}{rcl}
\dim \lstab(y) & \leq & q^2 \sgen - (q^2-q)\\
& = & q^2(\dim G - \dim X + m)-(q^2-q)\\
& = & q^2(\dim G - \dim X) + (m-1)q^2 + q.
\end{array}
$$
It follows that the orbit $G(q) \cdot y$ has dimension at least
$q^2 \dim X - (m-1)q^2 - q$, so that $X(q)$ has expected quotient dimension at most $(m-1)q^2+q$.
\end{proof}

\subsection{The stabilizer of $C$ in $G(q)$}

\subsubsection{The statement}

Recall from Section~\ref{sec:defC} the definition of $C$.

\begin{prop}
\label{pro:N}
Let $N_{\GL(V \otimes E)}(C):=\{g\in \GL(V\otimes E)\,:\, g\cdot C=C\}.$
We have
$$N_{\GL(V \otimes E)}(C) = \GL(V)^q\ltimes \fS_q.$$
\end{prop}

The proof of this proposition needs some preparation.

\subsubsection{Sum of subspaces of constant dimension}
\label{sec:sum}

The goal of this independent section is to prove some lemmas that will be useful to prove Proposition \ref{pro:N}.
We fix the following setting:
\begin{notation}
 \label{nota:linear-maps}
 Let $q$ be a positive integer, let $E_1,\ldots,E_q,F$ be vector spaces, let $\alpha_1,\ldots,\alpha_q:E_i \to F$
 be linear maps, and let $d_1,\ldots,d_q$ be integers such that $0 \leq d_i \leq \dim E_i$. Denote by $S$ the sum
 of the subspaces $\im\ \alpha_i$ for those $i$ such that $d_i=\dim E_i$.
\end{notation}
\noindent
We will analyse when it occurs that the dimension of $\sum \alpha_i(U_i)$
does not depend on the vector subspaces $U_i \subset E_i$ of dimension $d_i$.

\begin{lemma}
 \label{lem:image}
 Let $\alpha:E \to F$ be a linear map between finite dimensional vector spaces, and let $d$ be an integer between $0$ and $\dim E$.
 The set of all linear subspaces in $F$ of the form $\alpha(U)$ for $U \subset E$ a subspace of dimension $d$ is the
 set of all linear subspaces of $\im\ \alpha$ of dimension between $\max(0,d-\dim \ker \alpha)$ and
 $\min(d,\rk\ \alpha)$.
\end{lemma}
\noindent
The proof of Lemma~\ref{lem:image} will be omitted.

\begin{lemma}
 \label{lem:dim-interval}
 Let $q,\alpha_i:E_i \to F$ and $d_i$ be as in Notation \ref{nota:linear-maps}. Let $V \subset F$
 be a linear subspace. Then the set of the dimensions of the subspaces $(\sum \alpha_i(U_i)) \cap V$,
 where $U_i$ is any subspace in $E_i$ of dimension $d_i$, is an integer interval.
\end{lemma}
\begin{proof}
 Let $j \in \{1,\ldots,q\}$ be a fixed integer, and let a $(q-1)$-uple $(U_i)_{i \neq j}$ of subspaces as in the lemma
 be fixed. By Lemma \ref{lem:image}, when $U_j$ varies among the subspaces of $E_i$ of dimension $d_i$,
 the set of all subspaces of the form $\sum_{i=1}^q \alpha_i(U_i)$ is the set of all
 subspaces containing $\sum_{i \neq j} \alpha_i(U_i)$, included in $\sum_{i \neq j} \alpha_i(U_i) + \im(\alpha_j)$,
 and of dimension belonging to a given integer interval.
 
 It follows that the dimensions of the subspaces $(\sum \alpha_i(U_i)) \cap V$ when $U_j$ varies
 are an integer interval. Letting $j$ vary in $\{1,\ldots,q\}$, we deduce the lemma.
\end{proof}

\begin{lemma}
 \label{lem:dim-intersection}
 Let $E$ be a vector space, let $d,d'$ be integers between $0$ and $\dim E$, and
 let $V \subset E$ be a fixed subspace of dimension $d$. Assume that the dimension of $V \cap W$, for $W \subset E$
 a subspace of dimension $d'$, does not depend on $W$. Then at least one of the following occurs:
 \begin{enumerate}[label=(\greek*)]
  \item $d=0$;
  \item $d=\dim E$;
  \item $d'=0$;
  \item $d'=\dim E$.
 \end{enumerate}
\end{lemma}
\begin{proof}
 The minimal dimension of $V \cap W$ is $\max(d+d'-\dim E,0)$ and its maximal dimension is
 $\min(d,d')$. The equality of these integers implies that one of the four cases holds.
\end{proof}

\begin{lemma}
 \label{lem:dim-sum}
 Let $\alpha_i:E_i \to F$ be as in Notation \ref{nota:linear-maps}. The dimension of
 $\sum \alpha_i(U_i)$ does not depend on the vector subspaces $U_i \subset E_i$ of dimension $d_i$
 if and only if for all $i$, one of the following holds:
 \begin{enumerate}[label=(\roman*)]
  \item $d_i=0$,
  \item $0<d_i$ and $\im\ \alpha_i \subset S$,
  \item $0 < d_i < \dim E_i$, $\alpha_i$ is injective, and $\im\ \alpha_i \not \subset S$,
 \end{enumerate}
 and $S$ and the subspaces $\im\ \alpha_i$ for $i$ in case $(iii)$ are in direct sum.
\end{lemma}
\begin{proof}
It is plain that the given conditions imply that the dimension of $\sum \alpha_i(U_i)$ does not depend on
the $q$-uple $(U_i)$.
Conversely, assume that this dimension is constant. Let $\overline{\alpha}_1$ be the composition
$E_1 \stackrel{\alpha_1}{\longto} F \longto F/\sum_{i \geq 2} \alpha_i(U_i)$. The fact that $\dim \sum \alpha_i(U_i)$
deos not depend on $U_1$ implies that the dimension of $U_1 \cap \ker \oalpha_1$ does not depend on $U_1$. We are thus
in one of the four cases of Lemma \ref{lem:dim-intersection}. Case $(\alpha)$ is case $(i)$ of our Lemma. Assume
we are in case $(\beta)$. This implies $(ii)$. Moreover, letting $\oalpha_i$ be the composition
$E_i \longto F \longto F/\im\ \alpha_1$ when $i \geq 2$, we may assume by induction that
the lemma is true for the linear maps $\oalpha_2,\ldots,\oalpha_k$. Since the last condition of the lemma for
$\alpha_1,\ldots,\alpha_q$ is equivalent to the same condition for $\oalpha_2,\ldots,\oalpha_q$, the lemma is proved
in this case.

Note that condition $(\alpha)$ or $(\beta)$ holds for one $q$-uple $(U_i)$ if and only it holds for all $(U_i)$.
Assume now that these conditions never hold. Then, for any $(U_i)$ we either have condition $(\gamma)$, which is equivalent
to $\alpha_1$ being injective and $\im\ \alpha_1 \cap \sum_{i \geq 2} \alpha_i(U_i) = \{ 0 \}$, or condition
$(\delta)$, which is equivalent to $\im\ \alpha_1 \subset \sum_{i \geq 2} \alpha_i(U_i)$.

If both cases $(\gamma)$ and $(\delta)$ occur,
we apply Lemma \ref{lem:dim-interval} to $V=\im\ \alpha_1$ and the linear maps
$\alpha_2,\ldots,\alpha_k$, and we deduce that the rank of $\alpha_1$ is at most $1$. Since $\alpha_1$ is injective
because case $(\gamma)$ occurs, we deduce that $\dim E_1=0$ or $\dim E_1=1$, so case $(\alpha)$ or $(\beta)$ occurs.

If only case $(\gamma)$ occurs, we deduce that $\im\ \alpha_1$ is in direct sum with $\sum_{i \geq 2} \im\ \alpha_i$.
If only case $(\delta)$ occurs, we deduce that $\im\ \alpha_1 \subset S$. In each case, the conclusion of the lemma
holds.
\end{proof}

\subsubsection{Proof of Proposition~\ref{pro:N}}

Let $g \in N_{\GL(V \otimes E)}(C)$.
As in the proof of Proposition \ref{prop:modality}, we consider $g$ as a $q \times q$
matrix $g=(g_{i,j})_{,1\leq i,j \leq q}$ with coefficients $g_{i,j}$ in $\lgl(V)$. We choose
a factor $\Fl(A;V)$ of $X$ and we let $a \in A$.

The fact that $g$ preserves $C$ implies that given $U_1,\ldots,U_q \subset V$ of dimension $a$, there exist
$V_1,\ldots,V_q \subset V$ of dimension $a$ such that
$g \cdot (U_1 \otimes \CC f_1 \oplus \cdots \oplus U_q \otimes \CC f_q)
= V_1 \otimes \CC f_1 \oplus \cdots \oplus V_q \otimes \CC f_q$.

This implies that for $j$ in $\{1,\ldots,q\}$, we have $\sum_i g_{i,j}(U_i)=V_j$. Let $j$ be fixed:
the dimension of $\sum_i g_{i,j}(U_i)$ is always $a$, and we may apply Lemma \ref{lem:dim-sum} to the linear
maps $g_{i,j}:V \to V$. Since we have
$d_i=a$ for all $i$, we have $S=\{0\}$, case $(i)$ does not occur, and case $(ii)$ implies $g_{i,j}=0$. If some
$g_{i,j}$ is not equal to $0$, it is in case $(iii)$ and therefore it is an isomorphism $V \to V$. The condition
that the images of the linear maps $g_{i,j}$ are in direct sum implies that there can be at most one $i$ in case $(iii)$.
On the other hand, there is at least one such $i$ since $g$ is invertible.
It follows that $g=(g_{i,j})$ is a monomial matrix with coefficients
in $\GL(V)$, proving the proposition.
\qed

\section{Proof of the main theorem}

In this section, we prove Theorem~\ref{main-theo}.
We come back to the situation of Section~\ref{sec:setting}.
In particular, we have $k=3$ and $c_{\lambda,\mu,\nu}=2$ and $X$ is defined by
\eqref{eq:defXlambda}.
We will prove that 
$\dim \Ho(X(q),\Li(q)^{\otimes p})^{G(q)}=\binom{p+q}{q}$.

\subsection{Proof in the case of expected quotient dimension $1$}
\label{sec:mod1}

In this section, we make the extra assumption that $\modx=1$. 

\step{Details on $G$ acting on $X$.}

By \cite[Theorem~3.2]{teleman} and our assumption $c_{\lambda,\mu,\nu}=2$, we have \vspace{-3mm}

\begin{equation}
\CC^2\simeq\Ho(X,\Li)^{G} \simeq \Ho(X^\rss(\Li) ,
\Li_{|X^\rss(\Li)})^G\,.\label{eq:TelX}
\end{equation}

\noindent
By \cite[Proof of Corollary~2.4]{sherman2015geometric},
$X^\rss(\Li)\quot G$ is isomorphic to $\PP^1$.
Let $\pi_X\,:\,X^\rss(\Li) \longto \PP^1$ be the quotient map.
Observe that the stabilizer in the linear group of any point in a
product of flag variety is connected, as an open subset of some vector
space. By Kempf's criterion \cite[Théorème 2.3]{drezet}, this implies that there exists a line
bundle $\cO_{\PP^1}(d)$ on  $\PP^1$ such that $\pi_X^*(\cO_{\PP^1}(d))$ is the
restriction of $\Li$ to $X^\rss(\Li)$. Hence \vspace{-3mm}
 
\begin{equation}
\Ho(X^\rss(\Li) , \Li_{|X^\rss(\Li)})^G\simeq
\Ho(\PP^1,\cO_{\PP^1}(d)).\label{eq:KeX}
\end{equation}

Combining \eqref{eq:TelX} and \eqref{eq:KeX}, we get $d=1$. Now, the
same arguments imply that
$$
\begin{array}{ccl}
  \Ho(X,\Li^{\otimes p})^G&\simeq&\Ho(X^\rss(\Li),\Li^{\otimes p})^G\\
&\simeq&\Ho(\PP^1,\cO(p))\\
&\simeq&S^p\CC^2.
\end{array}
$$
Since the linear map $S^p \Ho(\PP^1,\cO(1))\longto \Ho(\PP^1,\cO(p))$ is an isomorphism, we get:

\begin{lemma}
\label{lemm:Sp}
The linear map $S^p \Ho(X,\Li)^G\longto \Ho(X,\Li^{\otimes p})^G$
is an isomorphism.
\end{lemma}

\step{Details on $N_{G(q)}(C)$ acting on $C$.}

It is well-known that the symmetric functions on $q$ variables $x_1,\ldots,x_q$ form a polynomial
algebra generated by the elementary symmetric functions $e_k$, where $e_k$ is the coefficient of $u^k$ in the
polynomial $\prod_{i=1}^q (x_iu+1)$. Writing $\PP^1$ as the union of two affine lines, one deduces that
$\bigoplus_p \Ho({(\PP^1)}^q,{\boxtimes}^{{}^{\scriptstyle{q}}} \cO_{\PP^1}(p))^{\fS_q}$
is a polynomial algebra generated by $(c_{k})_{0 \leq k \leq q}$,
where $c_{k}$ is the coefficient in $u^kv^{q-k}$
of the product $\prod_{i=1}^q (x_iu+y_iv)$. Here $(x_i,y_i)$ are sections of $\cO_{\PP^1}(1)$
on the $i$-th factor $\PP^1$:
\begin{lemma}
\label{lem:algC}
The algebra $\displaystyle \bigoplus_p \Ho({(\PP^1)}^q,{\boxtimes}^{{}^{\scriptstyle{q}}} \cO_{\PP^1}(p))^{\fS_q}$ is
freely generated by
$$\Ho({(\PP^1)}^q,{\boxtimes}^{{}^{\scriptstyle{q}}} \cO_{\PP^1}(1))^{\fS_q}=\Ho(\PP^q,\cO_{\PP^q}(1)).$$
\end{lemma}

\medskip

Recall that $C\simeq X^q$ and $N_{G(q)}(C)\simeq G^q\ltimes \fS_q$.
Hence, the previous step implies that 
\begin{equation}
\label{equa:CmodN}
C^\rss(\Li(q)_{|C})\quot N_{G(q)}(C)=(X^\rss(\Li)\quot G)^q\quot
\fS_q=(\PP^1)^q\quot \fS_q=\PP^q.
\end{equation}
Let
$\pi_\fS\,:\, (\PP^1)^q\longto\PP^q$ and
$\pi_N\,:\, C^\rss(\Li(q)_{|C})\longto\PP^q$ be the quotients map by
$\fS_q$ and $N_{G(q)}(C)$, respectively. The isomorphism of Lemma \ref{lem:algC} yields also
$\pi_\fS^* \cO_{\PP^q}(1)=\cO_{\PP^1}(1)^{\boxtimes q}$.
Now, Step 1 implies that $\pi_N^*(\cO_{\PP^q}(1))=\Li(q)_{|C}$.
Then, we have
\begin{equation}
\label{equa:sectionsC}
\begin{array}{ccl}
  \Ho(C,\Li(q)^{\otimes p})^{N_{G(q)}(C)}&\simeq&\Ho(C^\rss(\Li(q)_{|C}),\Li(q)^{\otimes p})^{N_{G(q)}(C)}\\
&\simeq&\Ho((\PP^1)^q,\cO_{\PP^1}(p)^{\boxtimes q})^{\fS_q}\\
&\simeq&S^p \Ho((\PP^1)^q,\cO_{\PP^1}(1)^{\boxtimes q})^{\fS_q}\\
&\simeq&S^p \Ho(C,\Li(q))^{N_{G(q)}(C)},
\end{array}
\end{equation}
where the third isomorphism comes from Lemma \ref{lem:algC}.

\step{Details on $G(q)$ acting on $X(q)$.}

Let $\pi_q\,:\,X(q)^\rss(\Li(q))\longto X(q)^\rss(\Li(q))\quot G(q)$
be the quotient map.

\begin{lemma}
\label{lem:pisurj}
The map $\pi:C^{\rm ss}(\Li(q)) \quot N_{G(q)}(C) \longto X(q)^{\rm ss}(\Li(q)) \quot G(q)$
in \eqref{def-pi} is surjective.  
\end{lemma}

\begin{proof}
The morphism $\pi$ is well defined by \eqref{eq:1}.
By properness, it is sufficient to prove that it is dominant. 
First observe that $\iota_q^*(\Li(q)_{|C})=\Li^{\boxtimes q}$ on $X^q$.
On the one hand, by \eqref{equa:CmodN},
$\dim C^\rss(\Li(q))\quot N_{G(q)}(C)=q$.
On the other hand, Proposition~\ref{prop:modality} and the assumption $\modx=1$
imply that $\modxq\leq q$.
Then, Proposition~\ref{prop:modmax} implies
$\dim X(q)^\rss(\Li(q))\quot G(q)\leq q$.  

Now, Lemma~\ref{lem:pifinite} implies that $\pi$ is surjective being
proper and finite.
\end{proof}

\begin{lemma}
\label{lem:res-injective}
For any $p\geq 0$, the restriction map
$$
\Ho(X(q),\Li(q)^{\otimes p})^{G(q)}\longto
\Ho(C,\Li(q)^{\otimes p})^{N_{G(q)}(C)}
$$
is injective.
For $p=1$, it is an isomorphism.
\end{lemma}

\begin{proof}
The last assertion follows from the first by the equality of the dimensions which follows
from \eqref{equa:cxq} and Theorem \ref{th:Iken}.

Let $\sigma \in \Ho(X(q),\Li(q)^{\otimes p})^{G(q)}$ be such that its restriction to $C$ is zero.
Let $x\in X(q)$: we show that $\sigma(x)=0$.
If $x$ is unstable, then by definition this means that any invariant
section vanishes at $x$.
Assume that $x$ is semistable, and set $\xi=\pi_q(x)$.

Pick $x_0$ in the closed $G(q)$-orbit in
$\overline{G(q) \cdot x}\cap  X^{\rm ss}(G(q),\Li(q))$.
By semi-stability, there exists a positive integer $k$ such that  the
stabilizer $G(q)_{x_0}$ acts trivially on $\Li^{\otimes k}_{x_0}$. It follows that the character of
$G(q)_{x_0}$ which defines the $G(q)$-linearized line bundle $\Li^{\otimes k}_{|G(q) \cdot x_0}$ is the trivial
character, and $\Li^{\otimes k}_{|G(q) \cdot x_0}$ is the trivial
$G(q)$-linearized line bundle on $G(q) \cdot x_0$.

On the other hand, by \cite[Theorem 1.10]{GIT}, the fiber $\pi_q\inv(\xi)$ is affine, and by \cite[Theorem 1]{bialynicki},
the stabilizer $G(q)_{x_0}$ is reductive.
We can therefore apply \cite[Lemma 2.1]{Br:PicVarSphe} (note that the normality assumption is not used in the 
proof of this Lemma) or \cite[Corollary 6.4]{bass}, and conclude that the
restriction of $\Li^{\otimes k}$ to $\pi_q\inv(\xi)$ is trivial.

Hence $\sigma^{\otimes k}$ can be viewed as a regular constant function on
$\pi_q\inv(\xi)$.
But Lemma~\ref{lem:pisurj} implies that $C$ intersects
$\pi_q\inv(\xi)$. Hence $\sigma^{\otimes k}$ vanishes on $\pi_q\inv(\xi)$. Finally
$\sigma^{\otimes k}$ and $\sigma$ vanish identicaly on
$\pi_q\inv(\xi)$. In particular $\sigma(x)=0$.                                                                                                                                                                  
\end{proof}

\step{Conclusion.}

Consider the following commutative diagram

\begin{equation}
\label{comm-diagram}
\begin{tikzcd}
S^p\Ho(X(q),\Li(q))^{G(q)}\ar[r,"\simeq"] \ar[d,"product" left] &
S^p\Ho(C,\Li(q))^{N_{G(q)}(C)}\ar[d,"\simeq"] \\
\Ho(X(q),\Li(q)^{\otimes p})^{G(q)} \arrow[hookrightarrow]{r} &
\Ho(C,\Li(q)^{\otimes p})^{N_{G(q)}(C)}
\end{tikzcd}
\end{equation}
The top horizontal map is an isomorphism and the bottom one is injective by Lemma \ref{lem:res-injective}.
The right vertical map is an isomorphism by \eqref{equa:sectionsC}.
It follows that the product map is an isomorphism. 

\begin{coro}
\label{coro:Pq}
 The GIT-quotient $X(q)^{\rm{ss}}(\Li(q)) \quot G(q)$ is isomorphic to $\PP^q$.
\end{coro}
\begin{proof}
By definition, this quotient is $\Proj$ of the algebra
$\bigoplus_p \Ho(X(q),\Li(q)^{\otimes p})^{G(q)}$. Since the product map in \eqref{comm-diagram} is an isomorphism,
this algebra is the symmetric algebra on $\Ho(X(q),\Li(q))^{G(q)}$, which is a vector space of dimension $q+1$.
\end{proof}

\subsection{Reduction to the case of expected quotient dimension $1$}

Observe that $\Li$ is ample and has $G$-invariant sections, so it belongs to $\acgx$.
We proceed by induction on $n$, considering two
cases:
\medskip

\noindent \underline{Case 1:} $\Li$ belongs to the interior of
$\acgx$. 

Then, by Proposition
\ref{prop:modmax}, the dimension of $X^\rss(\Li) \quot G$ is equal to $\modx$.
By Theorem~\ref{th:Iken}, we have $\dim \Ho(X,\Li^{\otimes p})^G = p+1$.
By definition of $X^\rss(\Li) \quot G$, see \eqref{equa:git-quotient}, the dimension of $X^\rss(\Li) \quot G$
is the degree of this polynomial, namely $1$.

We deduce that $\modx=1$ and we are done by Section\ref{sec:mod1}.

\medskip
\noindent \underline{Case 2:} $\Li$ is in the boundary of $\acgx$.

By Proposition~\ref{prop:ample-cone} and the ampleness of $\Li$, there
exist an integer $r$ and
$I,J,K \subset \{1,\ldots,n\}$
of cardinality $r$ such that
\begin{equation}
 \label{equa:face}
 \sigma_I \cup \sigma_J \cup \sigma_K = [pt] \mbox{ and }
 \frac 1r (|\lambda_I| + |\mu_J| + |\nu_K|) = \frac 1n (|\lambda| + |\mu| + |\nu|).
\end{equation}

Then, since the product $\sigma_I \cup \sigma_J \cup \sigma_K$ is equal to the class of the point,
by multiplicativity of Littlewood-Richardson coefficients \cite{dw,ressayre-birational},
we have $2=\Clmn=c_{\lambda_I,\mu_J,\nu_K} \cdot c_{\lambda_\oI,\mu_\oJ,\nu_\oK}$,
where $\oI=\{1,\ldots,n\} \setminus I$ (and similarly for $\oJ$ and $\oK$). We may thus assume the equalities
$c_{\lambda_I,\mu_J,\nu_K}=2$ and $c_{\lambda_\oI,\mu_\oJ,\nu_\oK}=1$. By induction, we deduce
$c_{\lambda_I(p,q),\mu_J(p,q),\nu_K(p,q)}=\binom{p+q}{p}$. By Fulton's
conjecture as stated in
Theorem \ref{th:Fulton}, we have
$c_{\lambda_\oI(p,q),\mu_\oJ(p,q),\nu_\oK(p,q)}=1$. Thus, the proof in this case will be finished if we can prove
that
\begin{equation}
\label{equa:c-product}
 c_{\lambda(p,q),\mu(p,q),\nu(p,q)} = c_{\lambda_I(p,q),\mu_J(p,q),\nu_K(p,q)}
 \cdot c_{\lambda_\oI(p,q),\mu_\oJ(p,q),\nu_\oK(p,q)}.
\end{equation}

\medskip

Relation~\eqref{equa:c-product} is proved using multiplicativity again.
First, observe that $\lambda_I(p,q)$ is equal to the partition $\lambda(p,q)_{I_q}$, where
\begin{equation}
\label{equa:Iq}
I_q=\{(i_1-1)q+1,\ldots,i_1q,(i_2-1)q+1,\ldots,i_2q,\ldots,(i_r-1)q+1,\ldots,i_rq\}
\end{equation}
if $I=\{i_1,\ldots,i_r\}$. Note that Schubert classes in $\G(r,n)$ are parametrized by subsets $I$ of $\{1,\ldots,n\}$
as we did in Section \ref{sec:horn}, and also by partitions whose Young diagram is included in a $r \times (n-r)$
rectangle. The correspondance maps a subset $I=\{i_1 < i_2 < \ldots < i_r \}$ to the partition
$(i_r-r,\ldots,i_2-2,i_1-1)$. Therefore, the partition corresponding to
$I_q$ is $q(i_r-r),\ldots,q(i_r-r),\ldots,q(i_1-1),\ldots,q(i_1-1)$ (with each part being repeated $q$ times).

If $\alpha$ denotes the partition corresponding to the subset $I$, then the partition corresponding to the subset
$I_q$ is $\alpha(q,q)$. Thus, by Theorem \ref{th:Fulton} again, the equality
$$
\sigma_I \cup \sigma_J \cup \sigma_K = [pt] \in H^*(\G(r,n),\Z)
$$
implies the equality
$$
\sigma_{I_q} \cup \sigma_{J_q} \cup \sigma_{K_q} = [pt] \in H^*(\G(qr,qn),\Z).
$$
By multiplicativity of Littlewood-Richardson coefficients, \eqref{equa:c-product} holds.

\section{About the case $c_{\lambda,\mu,\nu}>2$}

A key point in our proof is Lemma~\ref{lem:res-injective} showing
that, under the assumption $c_{\lambda,\mu,\nu}=2$, the restriction
map
$$
\rho_C\,:\,\Ho(X(q),\Li(q))^{G(q)}\longto
\Ho(C,\Li(q))^{N_{G(q)}(C)}=S^q \Ho(X,\Li)^{G}
$$
is injective. 
The following example shows that $\rho_C$ is not always injective.

\begin{exple}
   This example is mainly due to P.~Belkale \cite[Example
   3.7]{Belk:explec6}. For $G=\GL_8(\CC)$, consider
   $\lambda=\mu=(3,3,2,2,1,1)$ and $\nu=(4,4,4,3,3,2,2,2)$. We have
   $c_{\lambda,\mu}^\nu=6$.
Consider the Littlewood-Richardson polynomial $P_{\lambda,\mu}^\nu$
(see \cite{DW:LRpol}) such that for any $q\in\ZZ_{\geq 0}$,
$c_{q\lambda,\,q\mu}^{q\nu}=P(q)$.
This Littlewood-Richardson coefficient is obtained as the dimension of
a space of $G$-invariant sections on
$X=\Fl(2,4,6;\CC^8)^2\times\Fl(3,5;\CC^8)$.
It is easy to check that there exists $x$ in $X$ whose isotropy
group consists in the homotheties. Then $\modx=6$.
As a consequence the degree of  $P_ {\lambda,\mu}^\nu$ is
at most 6.
Using Buch's calculator \cite{Buch:calc}, one obtains that $P(0)=1$,
$P(1)=6$,
$P(2)=22$, $P(3)=63$, $P(4)=154$, $P(5)=336$ and $P(6)=672$.
Using Lagrange interpolation, one gets
$$
P_ {\lambda,\mu}^\nu(q)=\frac 1 {720}q^6+\frac 1 {48}q^5+\frac {23}
{144}q^4+\frac {35}{48}q^3+\frac {331} {180}q^2+\frac 9 {4}q+1,
$$
which indeed has degree $6$. In particular,  the map $\rho_C$  is not 
injective for $q$ big enough, since $\dim(S^q
\Ho(X,\Li)^{G})=\dim(S^q\CC^6)=\binom{q+5}{5}$ is a polynomial
function in $q$ of degree $5$.

Note that similarly, one gets
$$
P_ {\lambda',\mu'}^{\nu'}(k)=\frac {5} {2}(k^2+k)+1.
$$
\end{exple}


\bibliographystyle{amsalpha}
\bibliography{bidilatations}

\end{document}